\documentclass[sn-mathphys-num]{sn-jnl}

\usepackage{graphicx}%
\usepackage{multirow}%
\usepackage{amsmath,amssymb,amsfonts}%
\usepackage{amsthm}%
\usepackage{mathrsfs}%
\usepackage[title]{appendix}%
\usepackage{xcolor}%
\usepackage{textcomp}%
\usepackage{manyfoot}%
\usepackage{booktabs}%
\usepackage{algorithm}%
\usepackage{algorithmicx}%
\usepackage{algpseudocode}%
\usepackage{listings}%
\usepackage{enumerate}

\theoremstyle{thmstyleone}%
\theoremstyle{thmstyletwo}%
\newtheorem{Tma}{Theorem}

\newtheorem{Def}[Tma]{Definition}
\newtheorem{Lema}[Tma]{Lemma}
\newtheorem{Propo}[Tma]{Proposition}
\newtheorem{Coro}[Tma]{Corollary}

\theoremstyle{thmstylethree}%

\def\log{\operatorname{log}}

\def \res{\operatorname{res}}

\def\O{\mathcal{O}}

\newcommand\oo{
  \mathchoice
    {{\scriptstyle\mathcal{O}}}    {{\scriptstyle\mathcal{O}}}    {{\scriptscriptstyle\mathcal{O}}}    {\scalebox{.7}{$\scriptscriptstyle\mathcal{O}$}}  }
\begin{document}

\title[Article Title]{Restricted analytic valued fields with partial exponentiation}


\author*[1]{\fnm{Leonardo} \sur{Ángel}}\email{jangel@pedagogica.edu.co}

\author*[2]{\fnm{Xavier} \sur{Caicedo}}\email{xcaicedo@uniandes.edu.co}


\affil*[1]{\orgdiv{Departamento de Matemáticas}, \orgname{Universidad Pedagógica Nacional}, \orgaddress{\street{Calle 72 No.
11-86}, \city{Bogotá}, \postcode{110221}, \country{Colombia}}}

\affil[2]{\orgdiv{Departamento de Matemáticas}, \orgname{Universidad de los Andes}, \orgaddress{\street{Cra. 1. No.
18A-10}, \city{Bogotá}, \postcode{111711},  \country{Colombia}}}


\abstract{Non-archimedean models of the theory of the real ordered field with restricted analytic functions may not support a total exponential function, but they always have partial exponentials defined in certain convex subrings. On face of this, we study the first order theory of non-archimedean ordered
valued fields with all restricted analytic functions and an
exponential function defined in the valuation ring, which extends the
restricted analytic exponential. We obtain model completeness and
other desirable properties for this theory. In particular, any model embeds in a model where the partial exponential extends to a total one.}

\keywords{Restricted analytic function, ordered valued field, convex valuation, exponential function, model complete, weak o-minimality}


\pacs[MSC Classification]{03C10, 03C64, 12J25,12L12}

\maketitle

\section{Introduction}\label{sec1}

It is well known that the theory $T_{an}$ of the ordered real
field with restricted analytic functions is model complete and o-minimal\
(see \cite{Gabrielov}, \cite{Van9}, \cite{Van8}), and remains so when
expanded to the theory $T_{an}(\exp )$ with the total exponential function
(see \cite{Van1}, \cite{Wilk1}, \cite{Haskell}). However, not every model of 
$T_{an}$ admits total exponentiation. In fact, Kuhlmann, Kuhlmann and Shelah
proved in \cite{Kuhshe} that no Hahn field $\mathbb{R}((t^{\Gamma }))$ of
well based power series with exponents in a nontrivial divisible ordered abelian group $\Gamma $ admits a total
exponential function, although these fields are natural models of $T_{an}.$
Similarly, it may be seen that the field of logarithmic transseries (see \cite{van0}, \cite[Apendix A]{Van6}) is a 
model of $T_{an}$ because it is a union of Hahn fields, but it is not closed under exponentiation.

On the other hand, it is easy to extend the restricted analytic exponential
of any non-archimedean model of $T_{an}$ to an exponential function of its natural valuation
ring. More generally, any proper convex unital subring of a non-archimedean
model of $T_{an}$ (automatically a valuation ring) supports such an
exponential function if and only if its residue field is an exponential
field (see Proposition \ref{lema10}).

Motivated by this observation and the fact that valuations have become a
useful tool in the search of ordered fields equipped with an exponential
function (see \cite{Kuh4},\cite{Kuhlibro}, \cite{Kuh5}, \cite{Van2}), we focus in this
paper on the study of the first order theory $T_{an}(\O$-$\exp )$
of structures of the form $(K,\O,\exp )$ where $K$ is a model of $%
T_{an}$, $\O$ is a proper convex unital subring of $K$, and $\exp :%
\O\rightarrow \O$ is a surjective homomorphism from the
additive group of $\O$ onto its multiplicative group of positive
units $(\O^{\times})^{>0}$, which agrees in the unit $K$-interval $I_K=[-1,1]_K$ with the restricted analytic
exponential of $T_{an}$.

We show that for any model $(K,\O,\exp )$ of this theory the residue field $\res(K)$ equipped with the induced exponential is in fact a model of $T_{an}(\exp )$. Conversely, for any ordered valued field $(K,\mathcal{O)},$ if $(\res(K),e)$ is a model of 
$T_{an}(\exp )$, $e$ induces an exponential $\exp $ in $\O$ such that $(K,\O,\exp )$ is a model of $T_{an}(\O$-$\exp )$.

Next, following some ideas used to obtain quantifier elimination for $T_{an}(\exp ,\log)$ in \cite{Van1} and for $T_{convex}$ in \cite{Van10}, we prove our main results:\\

\textit{The theory }$T_{an}(\O$-$\exp )$\textit{ is model complete, complete, and weakly o-minimal. Moreover, it has
quantifier elimination when expanded with the partial inverse} $\log$
\textit{of} $\exp $,\textit{ and} $\ ^{-1},$ $\sqrt[n]{},$ $n\geq 1.$\\

Completeness follows from the existence of a prime model, and
quantifier elimination follows from the existence of closures of substructures. Another consequence of model completeness and the existence of a prime model is that
any model of this theory may be embedded in one
where the partial exponential\ of the valuation ring may be extended to a
total exponential. From this, weak o-minimality follows readily by the o-minimality of $T_{an}(\exp)$ and a result of Baizhanov \cite{Baijanov,Baijanov2}.

The structure of the paper is as follows. In Section 2, we include some basic facts about valued ordered fields and some useful
results on the theories $T_{an}$, $T_{an}(\exp ),$ and the theory $T_{convex} $ introduced by van den Dries and Lewenberg in \cite{Van10}. In
Section 3 we introduce the theory $T_{an}(\O$-$\exp )$ and prove it is model complete in Section 4. In Section 5, we prove quantifier elimination, the existence of a prime
model, and weak o-minimality.

For the general notions and facts of model theory, we refer the reader to 
\cite{Chang, Hod}, and for those of valued fields to \cite{End, Eng, main}.

\section{Preliminaries}

Throughout this paper, all rings considered will be commutative and unital. For a ring $A$, $A^{\times }$ will denote the multiplicative group of units of $A$, and we will indulge with the common abuse of using $A$ to denote the additive group of $A$.

\subsection{Valued fields, ordered fields and convex valuations}

\noindent A \textit{valued field} is a field $K$ with a \textit{valuation ring} $\O$, that is a necessarily local subring such that for any $x\in K^{\times}$ we have $x\in \O$ or $x^{-1}\in \O$. The natural group homomorphism $v:K^{\times}\rightarrow K^{\times}/\O^{\times}$ is the corresponding \textit{valuation} on $K$, and $v(K^\times)$ is called its \textit{value group}. This becomes an ordered group under the definition: $[x]\leq [y]$ if and only if $y/x\in \O$. The ring $\O$ may be recovered from the valuation map, and $v$ may be characterized, up to isomorphism, as a surjective homomorphism from $K^{\times}$ into an ordered abelian group $G$ satisfying $v(x+y)\geq \min\{v(x),v(y)\}$ for $x+y\neq 0$. We denote $\oo$ the unique maximal ideal of $\O$, and $\res(K)$ the \textit{residue field}, the image of the natural quotient map $\res:\O\rightarrow \O / \oo$. For detailed discussion of this concepts see \cite{End, Eng, main}.

Notice that if $\O\subseteq \O'$ are valuation rings in $K$, then $\oo'\subseteq \oo$. On the other hand, if $(K,\O)\subseteq (K',\O')$ \footnote{We use $\subseteq$ also to denote the substructure relation.} are valued fields then $\O=K\cap \O'$, $\O^{\times}=K^{\times}\cap \O'^{\times}$ and $\oo=K\cap \oo'$. Thus, up to isomorphism, $v(K^{\times})\subseteq v'(K'^{\times})$, $\res(K)\subseteq \res'(K')$, and the maps $v$ and $\res$ become the restrictions of $v'$ and $\res'$, respectively.

Given an ordered field $K$, any convex subring $\mathcal{O}$ of $K$ is a valuation ring of $K$ whose induced valuation $v$ satisfies 
\begin{equation*}
x\leq y\rightarrow v(y)\leq v(x),  \text{   for all } x,y\in K^{>0}.
\end{equation*} 

\noindent
Reciprocally, a valuation ring with a valuation map satisfying the above inequality is convex and we say that the valuation is \emph{convex}. Thus, we define an \textit{ordered valued field} as an ordered field with a convex\
subring; equivalently, with a convex valuation. 

From \cite[Lemma 3.4, Theorem 3.8]{Kuh4} we have the following result on the structures
of ordered valued fields:

\begin{Lema}[\textbf{Lexicographic decompositions of }$K$]
\label{decomp1} Let $(K,\O)$ be an ordered valued field. Then the underlying additive group $K$ decomposes as $K=A\oplus A^{\prime }\oplus \oo,$ where $A$ is a group complement \footnote{We say that $M$ is a \textit{group complement} of $N$ in the field $K$ if $M$
is a subgroup of $K$ such that $K=M\oplus N$.} to $\mathcal{O}$ in $K$
and $A^{\prime }$ is a group complement to $\oo$ in $\mathcal{O}$ such that $A^{\prime }$ is isomorphic to $\res(K)$.

Additionally, if $(\mathcal{O}^{\times})^{>0}$ denotes the multiplicative group of positive units of $\O$, $1+\oo=\{1+\epsilon :\epsilon\in \oo\}$, and the multiplicative group $K^{>0}$ is divisible, then $%
K^{>0}=B\cdot B^{\prime }\cdot (1+ \oo),$ where $%
B $ is a group complement to $(\mathcal{O}^{\times})^{>0}$ in $K^{>0}$
and $B^{\prime }$ is a group complement to $1+ \oo$ in $(\mathcal{O}^{\times})^{>0}$ such that $B^{\prime }$ is isomorphic to $\res(K)^{>0}$.
\end{Lema}

\bmhead{The natural valuation on an ordered field }

The natural valuation $w$ of an ordered field $K$ is the one induced the convex hull of $\mathbb{Q}$ in $K$ that we will denote $O_w$. Its value group may be identified with the archimedean classes of $K$ and its maximal ideal $\oo_w$ consists of the infinitesimals of $K$. The following are useful properties of this valuation.

\begin{Lema}\label{natural valuation} Let $F\subseteq K$ be ordered fields with $F$ real closed, $x\in K\setminus F$ and $F(x)$ the simple field extension induced by $x$, then:
\begin{enumerate}
    \item[$(1)$] \emph{(\cite[Lemma 3.4]{Van1})} If $w(F(x)^{\times
})\not=w(F^{\times })$, there is $a\in F$ such that $w(x-a)\not\in w(F^{\times }).$

    \item[$(2)$] \emph{(\cite[Corollary 2.2.3]{Eng})} If $w(x)\not\in w(F^{\times }),$ then $w(F(x)^{\times })=w(F^{\times })\oplus \mathbb{Z}w(x)$.
\end{enumerate}
\end{Lema}

\subsection{The theories $T_{an}$ and $T_{an}(\exp)$}

\noindent Let $L_{an}$ be the language of ordered rings $\{<,0,1,+,-,\cdot
\} $ augmented by a new function symbol for each restricted analytic
function $f_{I}:\mathbb{R}^{n}\rightarrow \mathbb{R}$:
\begin{equation*}
f_{I}(X)=
\begin{cases}
f(X), & \text{if }X\in I^{n} \\ 
0, & \text{if }X\notin I^{n},
\end{cases}
\end{equation*}

\noindent where $f$ belongs to $\mathbb{R}\{X_{1},...,X_{n}\}$, the ring of
all real power series in $X_{1},..,X_{n}$ converging in a neighbourhood of
the real unit $n$-cube $I^{n},$ $I=[-1,1].$ Then $T_{an}=Th(\mathbb{R}_{an})$
, where $\mathbb{R}_{an}$ is the ordered field of real numbers with its
natural $L_{an}$-structure. By \cite{Gabrielov}, \cite{Van9}, \cite{Van8} we
know that $T_{an}$ is model complete and o-minimal \footnote{%
We recall that a theory in which an order is given or definable is called 
\textit{o-minimal} if in every model of this theory, each definable subset
is a finite union of points and intervals (cf. \cite{P-S})}. It follows that $
\mathbb{R}_{an}$ is the prime model of this theory. Moreover, expanding $\mathbb{R}_{an}$ with the multiplicative inverse $x\mapsto \dfrac{1}{x}$ for $x\not=0$ and roots $\sqrt[n]{}$ of non-negative elements for $n>1$ (defined 0 otherwise), then $T_{an\ast }=Th(\mathbb{R}_{an},^{-1},\sqrt[n]{})$ has quantifier elimination and a universal axiomatization. \\

\noindent
\textbf{Example 1.}
Consider the \textit{Hahn field} $\mathbb{R}((t^{\Gamma }))$, that is the ordered field of the generalized power series $\sum\limits_{\gamma\in \Gamma}a_{\gamma}t^{\gamma}$ with coefficients $a_{\gamma}$ in $\mathbb{R}$, exponents $\gamma$ in an ordered abelian group $\Gamma$ and well-ordered support $\{\gamma\in \Gamma:a_{\gamma}\neq 0\}$. From \cite{Van1} we know that for any divisible ordered abelian group 
$\Gamma $, $\mathbb{R}((t^{\Gamma }))$ has a natural expansion to a model of $T_{an}$. By model completeness, the same is true of the union of any chain of such fields.\\

The following is a further important property of the natural valuation $w$
in models of $T_{an}$. Assume $F\subseteq K$ are models of $T_{an}$, $x\in K\setminus F,$ and $F\langle x\rangle $ is the
definable closure of $F(x)$ in $K$, which by model completeness is the model of $%
T_{an} $ generated by $F(x)$. Then

\begin{Lema} \label{divhull}
\emph{(\cite[Corollary 3.7]{Van1})} If $x\not\in F$ then $w(F\langle x\rangle ^{\times })$ is the divisible
hull of $w(F(x)^{\times }).$
\end{Lema}

Combining this with Lemma \ref{natural valuation} and because $%
w(F^{\times })$ is already divisible we have:

\begin{Coro}
\label{divisible hull} If $%
w(x)\not\in w(F^{\times })$ then $w(F\langle x\rangle ^{\times })$ $%
=w(F^{\times })\oplus \mathbb{Q}w(x).$
\end{Coro}

Let $L_{an,\exp }$ be the language $L_{an}$ augmented with a unary function
symbol $\exp .$ From \cite{Van1}, we know that the theory $T_{an}(\exp )=Th(%
\mathbb{R}_{an},e^{x})$ is model-complete and o-minimal with prime model $(%
\mathbb{R}_{an},e^{x}).$ Moreover, the theory $T_{an}(\exp ,\log)$ obtained by extending $T_{an}(\exp )$ with the defining axiom for the inverse $\log$ of $\exp$, $(x>0\rightarrow \exp (\log x)=x)\wedge (x\leq 0\rightarrow \log%
x=0),$ has quantifier elimination and a
universal axiomatization.\\

\noindent
\textbf{Example 2.}
The field of \textit{logarithmic-exponential transseries} $\mathbb{T}$ whose formal definition can be found
in \cite[Section 2]{Van2} or in \cite[Appendix A]{Van6}, yields an explicit
non-archimedean model of $T_{an}(\exp )$. As noticed in the introduction, the subfield $\mathbb{T}_{\log}$ of \textit{logarithmic transseries}; that is, transseries without exponential terms, is also a model of $T_{an}$ but not of $T_{an}(\exp )$.\\

\subsection{The theory $T_{convex}$}

Let $L$ be a language extending the language of ordered rings, $T$ a
complete o-minimal $L$-theory extending $Th(\mathbb{R},+,\times,0,1,<)$,
and $K$ a model of $T$. Following \cite{Van10}, we say that a convex subring $\O$ of $K$ is $T$-\textit{convex} if $f(\mathcal{O}%
)\subseteq \mathcal{O}$ for each 0-definable continuous \footnote{%
The notion of continuity is with respect to the order topology of $K$.}
function $f:K\rightarrow K$.

$T_{convex}$ denotes the theory of the pairs $(K,\mathcal{O})$ with $K$ a
model of $T$ and $\mathcal{O}$ a proper $T$-convex subring of $K$. As main result about the model theory of $T_{convex}$ we have:

\begin{Propo}
\emph{\label{model complete}\cite[Theorem 3.10, Corollaries 3.13 and 3.14]%
{Van10}} $T_{convex}$ is complete and weakly o-minimal \footnote{%
That is, in every model of this theory, each definable subset is a finite
union of convex subsets (cf. \cite{Dickmann}).}. Moreover:
\begin{enumerate}

\item[$(1)$] If $T$ is model complete, then $T_{convex}$ is model complete.

\item[$(2)$] If $T$ has quantifier elimination and is universally
axiomatizable, then $T_{convex}$ has quantifier elimination.
\end{enumerate}
\end{Propo}

In particular, $(T_{an})_{convex}$ is complete, model complete and weakly
o-minimal. In addition, $(T_{an\ast })_{convex}$ has quantifier elimination.
Further, as $T_{an}$ is polynomially bounded \footnote{%
A theory of ordered fields is polynomially bounded if for each model $K$ and
each $K$-definable function $f:K\rightarrow K$ there is $n\in \mathbb{N}$
and $a\in K^{>0}$ such that for all $r>a$ we have $|f(r)|<r^{n}$.} as shown
in \cite{Van8}, then by \cite[Proposition 4.2 and Remark 2.16]{Van10} we have:

\begin{Lema}
\label{tconvex1}For every proper convex subring $\mathcal{O}$ of a model $K$
of $T_{an}$, $(K,\mathcal{O})$ is a model of $(T_{an})_{convex}.$ Thus, $\O$ is closed under restricted analytic functions, and $\res(K)$ can be made into a model of $T_{an}$.
\end{Lema}

The following observation will be very useful in the proof of the main theorem. \medskip

\begin{Lema}\label{Lemma A} Let $(H,V)\subseteq (H\langle y\rangle,\O_1)$, $(H\langle y\rangle,\O_2)$ be models of $T_{convex}$ with $\O_1\subseteq \O_2$,  $y\in \O_1^{\times }$ and $w(y)\not\in w(H^{\times })$ for the natural valuation in $H\langle x\rangle$, then $\O_1=\O_2.$
\end{Lema}
\begin{proof}
Write $X^{\O}=\res_{\O}(X)$ for the image
of $X$ under the residual map with respect to a given local ring ${\O}$. As $V=H\cap \O_1=H\cap \O_2$ we may identify $V^{V}=V^{\O_1}=V^{\O_2}.$ The conditions on $y$ imply $y^{\O_i}\not\in V^{\O_i}$ for $i=1,2$ because $y=h+\varepsilon $ with $h\in V,$ $\varepsilon
\in \oo_{i}$, implies $h\in \O_{i}^{\times}$ and thus $y=h(1+\varepsilon^{\prime})$ with $\varepsilon^{\prime}\in \oo_{i}\subseteq \oo_{w}$, but $1+\varepsilon^{\prime }$ is a unit in $\O_{w}$ which contradicts $w(y)\not\in w(H^{\times }).$ Thus for $h\in V$ we are left with the
situations:\medskip 
\begin{center}

$y^{\O_1}<h^{\O_2},$ $y^{\O_{2}}<h^{\O_{2}}\qquad \qquad y^{\O_{1}}>h^{\O_{1}},$ $y^{\O_{2}}>h^{\O_{2}}$

$y^{\O_{1}}<h^{\O_{1}},$ $y^{\O_{2}}>h^{\O_{2}}\qquad \qquad y^{\O_{1}}>h^{\O_{1}},$ $y^{\O_{2}}<h^{\O_{2}}$
\end{center}
\noindent and the last two are contradictory since $\res$ reflects strict order.
Thus $y^{\O_{1}}$ and $y^{\O_{2}}$ realize the same cut over $V^{V}$ and $\O_{1}=\O_{2}$ by [22, Lemma 5.2].
\end{proof}

\section{The theory $T_{an}(\mathcal{O}$-$\exp )$}

Let $L_{+}=L_{an,\mathcal{O},\exp }$ be the language $L_{an}$ expanded by a
unary relation symbol $\mathcal{O}$ and a unary function symbol $\exp $. From now on, we single out the unary function symbol $\mathbf{e}$ of $L_{an}$ to denote the function corresponding to the analytic exponential $e^x=\sum\limits_{n=0}^{\infty}\dfrac{1}{n!}x^n$ restricted to $[-1,1].$

\begin{Def} \label{def9}
$T_{an}(\mathcal{O}$-$\exp )$ is the $L_{+}$-theory of the structures $(K,%
\mathcal{O},\exp )$ where $K$ is a model of $T_{an}$, $\mathcal{O}$ is a
proper convex subring of $K$ and $\exp :\mathcal{O}\rightarrow \mathcal{O}$
is a function such for all $x,y\in \mathcal{O}$ we have:

\begin{enumerate}
\item[$E1.$] $\exp(x+y)=\exp x\cdot \exp y$;

\item[$E2.$] $\exp x=\mathbf{e}(x)$ if $|x|\leq 1$;

\item[$E3.$] if $x>n^{2}$, then $\exp x>x^{n}$, for $n\geq 0$;

\item[$E4.$] if $y>1$, there is $x\in \mathcal{O}$ such that $\exp x=y$.
\end{enumerate}

\noindent \textit{For convenience we set }$\exp x=0$\textit{\ in }$%
K\setminus \mathcal{O}.$
\end{Def}

It follows from the definitions that $\exp 0=\mathbf{e}(0)=1$, and for any $x\in 
\mathcal{O}$ 
\begin{equation*}
\exp x=\exp (x/2+x/2)=\exp (x/2)^{2}>0.
\end{equation*}
Moreover, $\exp (x-x)=\exp 0=1$ and $\exp (-x)=(\exp x)^{-1}$. By axiom
E3, taking $n=0$, we obtain that if $x>0$ then $\exp x>1$; thus, if $%
x,y\in \mathcal{O}$ with $x<y$, then $\exp x<\exp y$. Hence $\exp $ is
positive and strictly increasing in $\mathcal{O}$. \ Let $y\in (\mathcal{O}%
^{\times })^{>0}$. If $y>1$, by axiom E4, there is $x\in \mathcal{O}$ such
that $\exp x=y$. If $y<1$, then $y^{-1}>1$ and by axiom E4, there is $x\in 
\mathcal{O}$ such that $\exp x=y^{-1};$ equivalently, $\exp (-x)=y$. Thus, 
\begin{equation*}
\exp (\mathcal{O})=(\mathcal{O}^{\times })^{>0}.
\end{equation*}

\noindent
\textbf{Example 3.} Let $K$ be a non-archimedean model of $T_{an}$ and $%
\mathcal{O}_{w}=\mathbb{R}+\oo$ be the valuation ring corresponding to the
natural valuation of $K$ with $\oo$ the maximal ideal of $\mathcal{O}_{w}$.
Define $\exp :\mathcal{O}_{w}\rightarrow \mathcal{O}_{w}$ as 
\begin{equation*}
\exp (r+\epsilon )=e^{r}\mathbf{e}(\epsilon )\text{ \ for }r\in \mathbb{R}%
\text{ and }\epsilon \in \oo
\end{equation*}%
then $(K,\mathcal{O}_{w},\exp )$ is a model of $T_{an}(\mathcal{O}$-$\exp )$. It may be seen that this is the unique map in $\O_w$ satisfying E1 and E2, since for $x\in \O_w$ there is $n$ such that $|x/n|<1$, thus $\exp x=\exp(n(x/n))=(\mathbf{e}(x/n))^n$. In particular, this applies to each Hahn field $K=\mathbb{R}((t^{\Gamma }))$
with $\Gamma $ divisible, since it can be expanded to a model of $T_{an}$.\\

It follows from Lemma \ref{tconvex1}  that if $(K,\mathcal{O},\exp )$ is a
model of $T_{an}(\mathcal{O}$-$\exp )$, then $(K,\mathcal{O})$ is a model of 
$(T_{an})_{convex}.$ Thus, $\mathcal{O}$ contains the prime model $\mathbb{R}%
_{an}$ of $T_{an}$ by \cite[Remark 2.7]{Van10}. Moreover $\exp $ extends the real exponential because 
$\exp r=\mathbf{e}(r)=e^{r}$ for $r\in \lbrack -1,1]_{\mathbb{R}}$ by E2,
and by E1 this coincidence extends to\ all of $\mathbb{R}$. That is,

\begin{Lema}\label{lemanuevo9} If $(K,\O,\exp)$ is a model of $T_{an}(\O$-$\exp)$, then $(\mathbb{R}_{an},e^x)$ is a substructure of $(\mathcal{O},\exp )$.
\end{Lema}

As a consequence of the completeness of $(T_{an})_{convex}$ and the fact
that for any model $K=\mathbb{R}((t^{\Gamma }))$ of $T_{an}$ one has $\mathbf{e}(\oo)=1+\oo$ (see \cite{Kuh4}, \cite{Kuh5}), it follows that this
happens in any model $(K,\mathcal{O})$ of $(T_{an})_{convex}$.

Therefore, we obtain the following correspondence:

\begin{Propo} \label{lema10}
\label{lifting1} Let $K$ be a model of $T_{an}$ and $\mathcal{O}$ be a
proper convex subring of $K$. Then:

\begin{enumerate}

\item[$(1)$]  If $(K,\mathcal{O},\exp )$ is a model of $T_{an}(\mathcal{O}$-$\exp )$%
, $\exp $ induces a total exponential function $\exp ^{\prime }$ on $\res%
(K)$, and

\item[$(2)$]  If $(\res(K),e)$ is a model of $T_{an}(\exp )$, $e$ induces an
exponential $\exp $ in $\mathcal{O}$ such that $(K,\mathcal{O},\exp )$ is a
model of $T_{an}(\mathcal{O}$-$\exp )$ and $\exp'=e$.
\end{enumerate}
\end{Propo}

\begin{proof}
If $(K,\mathcal{O},\exp )$ is a model of $T_{an}(%
\mathcal{O}$-$\exp )$ and $x,y\in \mathcal{O}$ are such that $\res(x)=%
\res(y)$, then $x=y+\epsilon $ for some $\epsilon \in \oo$, and thus $\exp x=\exp y \cdot\exp \epsilon =\exp
y \cdot(1+\epsilon ^{\prime })$ with $\epsilon ^{\prime }\in \oo$. Therefore, $%
\res(\exp x)=\res(\exp y)$ and we may define $\exp ^{\prime
}:\res(K)\rightarrow \res(K)$ as 
\begin{equation*}
\exp ^{\prime }(\res(x))=\res(\exp x).
\end{equation*}%
Since $\res$ is an homomorphism of ordered rings then\emph{\ }$E1,E3$
hold for $(\res(K),\exp ^{\prime })$ and $\exp ^{\prime }$ is an
strictly increasing embedding from the additive group $\res(K)$ into
the multiplicative group $\res(K)^{>0}$, actually surjective because $%
\res(\exp (\mathcal{O}))$ $=\res((\mathcal{O}^{\times })^{>0})=%
\res(K)^{>0}.$ Finally, since $\res(K)$ is a model of $T_{an}$
with $\mathbf{e}_{\res(K)}(x)=\res(\mathbf{e}_{K}(x)),$ $E2$
holds and the structure $(\res(K),\exp ^{\prime })$ is a model of $%
T_{an}(\exp )$. Reciprocally, let $e$ be a total exponential in $\res(K)$. By
Lemma \ref{decomp1}, there is a group complement $%
A^{\prime }$ to $\oo$ in $\mathcal{O}$ isomorphic to additive $%
\res(K)$, via $\res$, and a group complement $B^{\prime }$ to $1+%
\oo$ in $(\mathcal{O}^{\times })^{>0}$ isomorphic to
multiplicative $\res(K)^{>0}$, via $\res$. Thus, $e$ induces $e^{\prime }:A^{\prime }\rightarrow B^{\prime }$, $e^{\prime }(r)=\res^{-1}(e(\res(r)))$ and $(K,\mathcal{O},\exp )$ becomes a
model of $T_{an}(\mathcal{O}$-$\exp )$ with $\exp (r+\epsilon )=e^{\prime
}(r)\mathbf{e}(\epsilon )$ for $r\in A^{\prime }$ and $\epsilon \in %
\oo$. Moreover, $\exp'=e$ by construction\footnote{The map $e^{\prime }:A^{\prime }\rightarrow B^{\prime }$ introduced in this proof is called a middle exponential by S. Kuhlmann in  \cite{Kuh4, Kuhlibro}.}.
\end{proof}
\medskip

\section{Model completeness of $T_{an}(\mathcal{O}$-$\exp )$}

We will prove in this section that the theory $T_{an}(\mathcal{O}$-$\exp )$
is model complete utilizing a criterion originally due to Sacks (Theorem
17.1 \cite{sacks}): $\ T$\emph{\ is model complete if and only if for any
embeddings\ }$\mathfrak{C}\subseteq \mathfrak{A,}$ $\mathfrak{C}\overset{%
\psi }{\rightarrow }\mathfrak{B}$ \emph{of models of T, with }$\mathfrak{B}$\emph{\ }$|\mathfrak{A}|^{+}$\emph{-saturated, }$\psi $\emph{\ may be extended to
an embedding }$\psi' :\mathfrak{A}\rightarrow $\emph{\ }$\mathfrak{B}$%
\emph{.}

Thus we must show

\begin{Propo}
\label{mainextension} Let $(E,\mathcal{O}_{E},\exp _{E}),$ $(K,\mathcal{O}%
,\exp ),$ $(K^{\ast },\mathcal{O}^{\ast },\exp ^{\ast })$ be models of $%
T_{an}(\mathcal{O}$-$\exp )$ with embeddings 
\begin{equation*}
\begin{array}{cccc}
(K,\mathcal{O},\exp ) &  & 
(K^{\ast },\mathcal{O}^{\ast },\exp ^{\ast })\\ 
\uparrow & \text{ }\nearrow _{\psi } &  &  \\ 
(E,\mathcal{O}_{E},\exp _{E}) &  &  & 
\end{array}%
\end{equation*}%
and the last model $|K|^{+}$-saturated. Then there is an embedding $\psi
^{\prime }:(K,\mathcal{O},\exp )\rightarrow (K^{\ast },\mathcal{O}^{\ast
},\exp ^{\ast })$ extending $\psi .$
\end{Propo}

To prove this proposition we consider two cases $\res(E)=\res(K)$
and $\res(E)\not=\res(K).$ In the following, $\exp$ and $\log$ will exclusively refer to the operations in the ring $\O$ and $\exp^{\ast}, \log^{\ast}$ to those in $\O^{\ast}$.

\subsection{Case $\res(E)=\res(K)$}

Since $K^{\ast }$ is $|K|^{+}$ saturated and $(T_{an})_{convex}$ is
model-complete, there is an embedding $\psi ^{\prime }:(K,\mathcal{O}%
)\rightarrow (K^{\ast },\mathcal{O}^{\ast })$ extending $\psi :(E,\mathcal{O}%
_{E})\rightarrow (K^{\ast },\mathcal{O}^{\ast })$. Moreover, for any $a\in 
\mathcal{O}$ there is by hypothesis $b\in \mathcal{O}_{E}$ such that $%
a=b+\epsilon $ with $\epsilon $ in $\oo,$ then since $\oo\subseteq I_{K},$%
\begin{equation*}
\exp a=\exp b\cdot \exp \epsilon =\exp _{E}b\cdot \mathbf{e}_{K}(\epsilon ),
\end{equation*}%
and since $\psi $ preserves exponential and $\psi ^{\prime }$ preserves
restricted analytic functions then 
\begin{equation*}
\psi ^{\prime }(\exp a)=\psi (\exp _{E}b)\psi ^{\prime }(\mathbf{e}%
_{K}(\epsilon))=\exp^{\ast } \psi (b)\cdot \mathbf{e}_{K^{\ast }}(\psi
^{\prime }(\epsilon ))=\exp^{\ast } \psi ^{\prime }(a),
\end{equation*}%
so that $\psi ^{\prime }$ embeds $(K,\mathcal{O},\exp )$ into $(K^{\ast },%
\mathcal{O}^{\ast },\exp ^{\ast })$.

\subsection{Case $\res(E)\not=\res(K)$}

This case will be reduced to the first. Consider an embedding of $(T_{an})_{convex}$ models $\varphi :(H,\mathcal{O}_{H})\overset{\backsimeq}{\rightarrow} (H^*,\O_{H^*})\subseteq (K^{\ast },\mathcal{O}^{\ast })$ maximal with respect to the following properties:\\

- $(E,\O_{E})\subseteq (H,\O_{H})\subseteq (K,\O)$ and $\O_H$ is $\log$-closed; that is, $\log((\O^{\times}_H)^{>0})\subseteq\O_H$.

- $\varphi$ extends $\psi$.

- $\varphi $ is $\log$-preserving.\\

\noindent
Clearly, such maximal extension exists by Zorn's Lemma. We will prove:\\

\textbf{Claim 1.} $\O_H$ is $\exp $-closed.

\textbf{Claim 2.}   $\res(H)=\res(K).$\\

\noindent
Which will put in the first case, since $(H,\O_H,\exp|_{\O_H})$ becomes a model of $T_{an}(\O$-$\exp)$ and $\varphi$ an embedding of these models.

The proof of these two claims will follow from a sequence of lemmas. The next one takes care of the first claim. Recall that $H\langle a\rangle $ denotes the
model of $T_{an}$ generated by the simple extension field $H(a)$ and $w$ denotes the natural valuation of $K$.

\begin{Lema}
\label{ext12} $\mathcal{O}_{H}$ is $\exp $-closed.
\end{Lema}

\begin{proof} By maximality of $(H,\mathcal{O}_{H})$, it is enough to
show that if $x\in \mathcal{O}_{H}$ and $\exp x\in \mathcal{O}\smallsetminus \mathcal{O}_{H}$ then $\mathcal{O}_{H\langle \exp
x\rangle }=\mathcal{O}\cap H\langle \exp x\rangle $ is $\log$-closed and $\varphi $ may be extended to $(H\langle \exp x\rangle ,%
\mathcal{O}_{H\langle \exp x\rangle })$ as a $(T_{an})_{convex}$ embedding
preserving $\log.$ We show first%
\begin{equation*}
w(\exp x)\not\in w(H^{\times}),
\end{equation*}%
otherwise, $\exp x=b(1+\epsilon )$ with $b\in H^{>0},$ $\varepsilon \in \oo_{w}\subseteq I_{K}\subseteq \mathcal{O}$. As $1+\epsilon $
is a positive unit in $\mathcal{O},$ then $b\in (\mathcal{O}^{\times})^{>0}\cap H=(\mathcal{O}
_{H}^{\times})^{>0}$. Let $d=\log(b)\in \mathcal{O}_{H},$ then $\exp x=(\exp
d)(1+\varepsilon )$ and thus $\exp (x-d)=1+\varepsilon .$ Hence, $x-d=\log(1+\varepsilon )\in I_{H}$ (since $\log%
(1+\varepsilon )\leq |\varepsilon |<1)$; and, $\exp (x-d)=\mathbf{e}%
(x-d)\in \mathcal{O}_{H}$ because $\mathcal{O}_{H}$ is closed under the
restricted analytic exponential. As $\exp d=b\in \mathcal{O}_{H},$ this
implies $\exp x\in \mathcal{O}_{H},$ a contradiction.

We have then that $w(H\langle \exp x\rangle^{\times} )=w(H^{\times})+\mathbb{Q}w(\exp x)$ by
Corollary \ref{divisible hull}. Thus, any $g\in \mathcal{O}_{H\langle \exp
x\rangle }^{>0}$ has the form $g=h(\exp x)^{q}(1+\varepsilon )$ with $h\in
H, $ $q\in \mathbb{Q}$ and $\varepsilon \in \oo_{w}.$ As $1+\varepsilon $
and $(\exp x)^{q}$ are positive units in $\mathcal{O}_{H\langle \exp x\rangle }$ then 
$h\in \mathcal{O}_{H\langle \exp x\rangle }\cap H=\mathcal{O}_{H}$ and,
moreover, $h$ is a positive unit in $\mathcal{O}$ whenever $g$ is. Hence, for 
$g\in (\mathcal{O}_{H\langle \exp x\rangle }^{\times })^{>0}:$ 
\begin{equation*}
\log g=\log h+qx+\log(1+\varepsilon )\in \mathcal{O}%
_{H\langle \exp x\rangle }
\end{equation*}%
because $\log h,$ $x\in \mathcal{O}_{H},$ and $\mathcal{O}_{H\langle
\exp x\rangle }$ is divisible and closed under analytic $\log(1+x)$ restricted to $\oo_w \subseteq [-1/2,1/2]_{K}.$ This shows $\O_H\langle \exp(x)\rangle$ is $\log$-closed.

The map $\varphi $ extends to a $T_{an}$-embedding $\varphi ^{\prime
}:H\langle \exp x\rangle \overset{\backsimeq}{\rightarrow} H^{\ast }\langle \exp ^{\ast }\varphi
(x)\rangle \subseteq K^*$ such that $\varphi ^{\prime }(\exp x)=\exp ^{\ast }\varphi (x)$
by o-minimality of $T_{an},$ because $\exp x$ has the same order type over $%
H $ as $\exp^{\ast} \varphi (x)$ has over $H^{\ast }.$ To check this is enough to
consider parameters $h\in (\mathcal{O}^{\times })^{>0}\cap H=(\mathcal{O}%
_{H}^{\times })^{>0}$ because $\exp x$ is a positive unit\ in $\mathcal{O}$
and similarly for $\exp ^{\ast }\varphi (x)$ in $\mathcal{O}^{\ast }$. But $%
\log h$ exists in $\mathcal{O}_{H}$ and thus%
\begin{eqnarray*}
\exp x<h\Leftrightarrow x<\log h\Leftrightarrow \varphi (x)<\varphi (%
\log h)=\log^{\ast }\varphi (h) \\
\Leftrightarrow \exp ^{\ast }\varphi (x)<\varphi (h).
\end{eqnarray*}%

Moreover, $\varphi ^{\prime }$ is $\log$-preserving because $\varphi $ preserves~log and $\varphi ^{\prime }$
preserves $\log(1+x)$ restricted to $\oo_w$:
\begin{eqnarray*}
\varphi ^{\prime }(\log g) &=&\varphi (\log h)+q\varphi
(x)+\varphi ^{\prime }(\log(1+\varepsilon )) \\
&=&\log^{\ast }\varphi (h)+q\log^{\ast }(\exp ^{\ast }\varphi
(x))+\log^{\ast }(1+\varphi ^{\prime }(\varepsilon )) \\
&=&\log^{\ast }(\varphi (h)\varphi ^{\prime }(\exp x)^{q})\varphi
^{\prime }((1+\varepsilon ))=\log^{\ast }(\varphi ^{\prime }(g)).
\end{eqnarray*}%

Finally, Lemma \ref{Lemma A} may be applied to  $(H^*,\O_{H^*})\subseteq (H^*\langle \exp^*\varphi(x)\rangle,\O_1)$, $(H^*\langle \exp^*\varphi(x)\rangle,\O_2)$ with $\O_1=\varphi'(\O_H\langle \exp x\rangle)$, $\O_2=H^*\langle \exp^* \varphi(x)\rangle \cap \O^*$, because $\exp^* \varphi(x)$ is a unit of $\O_1$ and $w(\exp^* \varphi(x))\notin w(H^{*\times})$, to conclude that both rings are equal and  $\varphi ^{\prime }:(H\langle \exp x\rangle ,\mathcal{O}%
_{H\langle \exp x\rangle })\rightarrow $ $(K^{\ast },\mathcal{O}^{\ast })$
is a $(T_{an})_{convex}$-embedding.

\end{proof}

The second claim will follow from the next three
lemmas. 

\begin{Lema}
\label{ext11} If $x\in K\smallsetminus H$ then $w(H(x)^{\times
})\not=w(H^{\times })$.
\end{Lema}

\begin{proof} Assume $w(H(x)^{\times })=w(H^{\times })$ for the sake of
contradiction, then $w(H\langle x\rangle ^{\times })=w(H^{\times })$ by Lemma \ref{divhull} because 
$w(H^{\times })$ is already divisible. Thus any $a\in \mathcal{O}_{H\langle
x\rangle }^{>0}$ has the form $a=h(1+\epsilon )$ with $h\in H^{\times }$ and 
$\epsilon \in \oo_{w}$, so $(1+\epsilon )\in \mathcal{O}_{H}.$
Since $(1+\epsilon )$ is positive and invertible in $\mathcal{O}_{H}$ then $%
h\in \mathcal{O}_{H\langle x\rangle }^{>0}\cap H=\mathcal{O}_{H}^{>0}.$
Therefore, $H\langle x\rangle $ is $\log$-closed because $\log%
(a)=\log(h)+\log(1+\varepsilon ),$ which belongs to $H\langle
x\rangle $ since $\log(h)\in \mathcal{O}_{H}$ and $\mathcal{O}%
_{H\langle x\rangle }$ is closed under $\log(1+x)$ restricted to $\oo_w$.

Using model completeness of $(T_{an})_{convex}$ and Sacks criterium there is
a $(T_{an})_{convex}$-embedding $\varphi ^{\prime }:(H\langle x\rangle ,%
\mathcal{O}_{H\langle x\rangle })\rightarrow $ $(K^{\ast },\mathcal{O}^{\ast
})$ extending $\varphi ,$ which preserves $\log$ because $\varphi $
preserves log in $\mathcal{O}_{H}$ and $\varphi ^{\prime }$ preserves $\log(1+x)$ restricted to $\oo_w:$ 
\begin{equation*}
\varphi ^{\prime }(\log a)=\varphi (\log h)+\varphi ^{\prime }(%
\log(1+\varepsilon ))=\log^{\ast }\varphi (h)+\log^{\ast
}\varphi ^{\prime }(1+(\varepsilon ))=\log^{\ast }\varphi ^{\prime
}(g).
\end{equation*}%
This contradicts the maximality of $\varphi .$

\end{proof}

Using the two previous lemmas we have:

\begin{Lema}
\label{Nuevo}If $x\in (\mathcal{O}^{\times })^{>0}\ $and $\res(x)\not\in \res(%
\mathcal{O}_{H}\mathcal{)}$ then there is $a\in \mathcal{O}_{H}$ such that $%
\res(\log x-a)\not\in \res(H)$ and $w(\log x-a)\not\in w(\mathcal{O}%
_{H}^{\times})$.
\end{Lema}

\begin{proof} Notice that $\res(\log x-a)\not\in \res(O_H)$ for any $a\in \O_{H}$, otherwise, $\res(\log x)=\res(h)$ with $h\in \O_{H}$, $\log x=h+\varepsilon$ with $\varepsilon \in \oo$ and $x=\exp h\cdot \exp\varepsilon =(\exp h)\cdot (1+\varepsilon ^{\prime })=\exp h+\varepsilon ^{\prime \prime }$ with $\epsilon''$ in $\oo$, so $\res(x)\in \res(\O_H)$ by Lemma \ref{ext12}, a contradiction. 

For the second claim of the lemma, as $x\in \mathcal{O}\smallsetminus \mathcal{O}_{H}$ and $%
\mathcal{O}_{H}$ is $\exp$-closed then $\log x\not\in H$ and by Lemma \ref%
{ext11} $w(H(\log x)^{\times })\not=w(H^{\times }),$ then by Lemma \ref%
{natural valuation} there is $a\in H$ such that $w(\log x-a)\not\in
w(H^{\times}).$ If $a\in \mathcal{O},$ thus $a\in \mathcal{O}\cap H=\mathcal{O}_{H}$
and we are done. If $a\not\in \mathcal{O}$, then $a^{-1}\in %
\oo\subseteq \O$, it remains to show that $w(\log 
x-a^{-1})\not\in w(H^{\times}).$ Otherwise, $\log x-a^{-1}=h(1+\varepsilon )$
with $h\in H,$ $\varepsilon \in \oo_{w}\subseteq \mathcal{O}_{w}\subseteq 
\mathcal{O}$ \footnote{%
Note that it is not the case that $\mathcal{O}_{w}\subseteq \mathcal{O}_{H}.$%
}. As $1+\varepsilon $ is a unit in $\mathcal{O}$ then $h\in \mathcal{O}\cap
H=\mathcal{O}_{H},$ therefore, $\log x-a^{-1}=h+h\varepsilon \in 
\mathcal{O}_{H}+\oo,$ contradicting the first claim.

\end{proof}

The next lemma contradicts the maximality of $\varphi $\ and thus finishes the
proof of Claim 2: $\res(H)=\res(K).$ We follow the strategy of the proof of Lemma
4.4 in \cite{Van1} to build a $\log$-closed extension. Similar ideas have been used to prove Theorem 6.44 in \cite[Chapter 6]{Kuhlibro}, but we must allow here for the fact that only
positive units of $\mathcal{O}$ have a logarithm in our case.

\begin{Lema}
\label{ext13} If $x\in \mathcal{O}$ and $\res(x)\not\in \res(O_H)$ then there is
a sequence $(x_{n})$ in $(\mathcal{O}^{\times })^{>0}$ such that $%
w(x_{0}),w(x_{1}),...$ are $\mathbb{Q}$\emph{-}linearly independent over $%
w(H^{\times })$. Moreover, if $H^{\prime }=\cup _{n\in \omega }H\langle x_0,x_1,...,x_n\rangle$, then  $\mathcal{O}^{\prime }=\mathcal{O}\cap H' $ is $\log$-closed and $\varphi $ may be extended to a $%
(T_{an})_{convex}$ $\log$-preserving embedding $\varphi ^{\prime }:(H',\mathcal{O}^{\prime })$ $\rightarrow (K^{\ast },%
\mathcal{O}^{\ast })$.
\end{Lema}

\begin{proof} If $\res(x)\not\in \res(\mathcal{O}_{H})$ then $\res(x)\not=0$
thus $x$ is a unit of $\mathcal{O},$ and we may assume without loss of
generality $x>0,$ thus $x\in (\mathcal{O}^{\times })^{>0}$. Moreover, $%
x\not\in \mathcal{O}_{H}$ and we may define $x_{0}=|\log x-a|^{\pm 1},$
with $a$ chosen as in Lemma \ref{Nuevo} and the exponent chosen so that $w(x_{0})<0$%
. Clearly $|\log x-a|\in \mathcal{O}^{\times }$ because $\res(|\log%
x-a|)\not=0.$ By Lemma \ref{Nuevo} we may define inductively$\ x_{n+1}=|%
\log x_{n}-a_{n}|\ $so that for all $n$

\begin{equation*}
w(x_{n})\not\in w(H^{\times }),\text{ }\res(x_{n})\not\in \res(\mathcal{O}_{H})%
\text{ (hence, }x_{n}\in (\mathcal{O}^{\times })^{>0}),
\end{equation*}%
and it follows by induction that 
\begin{equation*}
w(x_{n})<w(\log x_{n})\leq w(x_{n+1})<0,
\end{equation*}%
because we have chosen $w(x_{0})<0,$ and assuming $w(x_{n})<0\ $we have $%
w(x_{n})<w(\log x_{n})<0$ because $x_{n}>k^{2}>k$ for all $k\in \mathbb{%
N}$ and $\log x_{n}\not\in \mathcal{O}_{w}$ ($\mathcal{O}_{w}$ is
closed under $\exp )$. Also, $w(\log x_{n})\leq w(x_{n+1})$, otherwise $%
w(x_{n+1})=\min (w(\log x_{n}),w(a_{n}))=w(a_{n})\in w(H^{\times }),$ a
contradiction. Moreover, $w(x_{n+1})$ $<0$, since $x_{n+1}\in \mathcal{O}%
_{w} $ would imply $\exp (\pm x_{n+1})\in \mathcal{O}_{w}^{\times }$ and so $w(x_{n})=w(\exp (\pm x_{n+1}))+w(\exp a_{n})=w(\exp a_{n})\in w(H^{\times}),$ a contradiction. 

To show the $\mathbb{Q}$\emph{-}linear independence of $%
w(x_{0}),w(x_{1}),... $ over \emph{\ }$w(H^{\times })$, assume there are $%
m<n $, rational numbers $q_{m+1},...,q_{n}$, and an element $b\in H$ such
that 
\begin{equation*}
w(x_{m})=\sum\limits_{i=m+1}^{n}q_{i}w(x_{i})+w(b).
\end{equation*}%
Equivalently, for some $c\in \mathcal{O}^{>0}$ with $w(c)=0$ we have $
x_{m}=cb\prod\limits_{i=m+1}^{n}x_{i}^{q_{i}}$. Thus, $\log x_{m}=\log c+\log b+\sum\limits_{i=m+1}^{n}q_{i}\log x_{i}$ and 
\begin{eqnarray*}
w(x_{m+1}) &=&w(\log x_{m}-a_{m})=w(\log c+\log b+\sum\limits_{i=m+1}^{n}q_{i}\log x_{i}-a_{m}) \\
&=&w(\log b-a_{m})\in w(H^{\times })\text{ }
\end{eqnarray*}%
because $w(\log c+\sum\limits_{i=m+1}^{n}q_{i}\log%
x_{i})>w(x_{m+1}),$ since $w(\log x_{i})<0\leq w(\log c)$ for $%
i=m+1,...,n.$ But this contradicts that $w(x_{m+1})\not\in w(H^{\times })$.

Now, define $H_{n}=H\langle x_{0},..,x_{n}\rangle ,$ then by induction and Corollary \ref{divisible hull}: 
\begin{equation*}
w(H_{n}^{\times })=w(H^{\times })\oplus \mathbb{Q}w(x_{0})\oplus \cdots
\oplus \mathbb{Q}w(x_{n});
\end{equation*}%
$w(H_{0}^{\times })=w(H^{\times })\oplus \mathbb{Q}w(x_{0})$ because $%
w(x_{0})\not\in w(H^{\times})$, and the displayed identity implies $w(x_{n+1})\not\in
w(H_{n}^{\times})$ by linear independence, thus $w(H_{n+1}^{\times})=w(H_{n}\langle
x_{n+1}\rangle^{\times})=w(H_{n}^{\times })\oplus \mathbb{Q}w(x_{n+1})=w(H^{\times
})\oplus \mathbb{Q}w(x_{0})\oplus \cdots \oplus \mathbb{Q}w(x_{n+1})$. From this we get that any $a\in H_{n}^{\times }$ has the form 
\begin{equation*}
a=b(1+\epsilon )\prod\limits_{i=0}^{n}x_{i}^{q_{i}}
\end{equation*}
for some $b\in H$, $\epsilon \in H_{n}$ with $w(\epsilon )>0,$ and rational
numbers $q_{0},...,q_{n}$.

\noindent
If $a\in \mathbf{(}\mathcal{O}^{\times }\mathcal{)}^{>0}\cap H_{n}$ then $%
b\in \mathbf{(}\mathcal{O}_{H}^{\times }\mathcal{)}^{>0}$ because $%
x_{i},1+\varepsilon $ are positive units in $\mathcal{O},$ so $\log b$
exists and belongs to $H$; hence, 
\begin{equation*}
\log a=\log b+\log(1+\epsilon )+\sum\limits_{i=0}^{n}q_{i}%
\log x_{i}\in \mathcal{O}\cap H_{n+1}
\end{equation*}%
because $\log(1+\epsilon )\in \mathcal{O}\cap H_{n}$.

Define $H^{\prime }=\cup _{n\in \omega }H_{n},$ $\mathcal{O}^{^{\prime }}=%
\mathcal{O\cap }H^{\prime }$ then the above means that $(H^{\prime },%
\mathcal{O}^{\prime })$ is $\log$-closed. It remains to show that $\varphi $
may be extended to a $(T_{an})_{convex}$ $\log$-preserving embedding $\varphi
^{\prime }:(H^{\prime },\mathcal{O}^{\prime })\rightarrow (K^{\ast },%
\mathcal{O}^{\ast })$.

Let $H^*=\phi(H)$ and use model completeness of $(T_{an})_{convex}$ to extend $\varphi $
to 
\begin{equation*}
\varphi _{0}:(H\langle x_{0}\rangle ,\mathcal{O\cap }H\langle x_{0}\rangle
)\approx (H^{\ast }\langle y_{0}\rangle ,\mathcal{O^{\ast }\cap }H^{\ast
}\langle y_{0}\rangle )\subseteq (K^{\ast },\mathcal{O}^{\ast })
\end{equation*}%
then $y_{0}=\varphi _{0}(x_{0})$ inherits the properties

\begin{equation*}
\text{ }w(y_{0})\not\in w(H^{\ast \times }),\text{ }w(y_{0})<0\text{, }%
y_{0}\in (\mathcal{O}^{\ast \times })^{>0}.
\end{equation*}%
Notice that $\varphi ^{-1}:(H^{\ast },\mathcal{O}_{H^{\ast }})\overset{\backsimeq}{\rightarrow} (H,%
\mathcal{O}_{H})\subseteq (K,\mathcal{O})$ is also maximal, therefore if we
define inductively $y_{n+1}=|\log y_{n}-\varphi (a_{n})|$ in $K^{\ast
}$ we obtain as for the $x_{n}$ the independence of the $w(y_n)$ over $w(H^{*\times})$ and 
\begin{equation*}
w(y_{n})\not\in w(H^{*\times }),w(y_{n})<w(\log y_{n})\leq w(y_{n+1})<0, \text{ }y_{n}\in (\mathcal{O}^{\times })^{>0}.
\end{equation*}%

Call $H^*_n=H^*\langle y_0,...,y_n \rangle$ and assume by induction hypothesis there is an embedding $\varphi _{n}:(H_{n},\mathcal{O}_{n})\overset{\backsimeq}{\rightarrow}
(H_{n}^{\ast },\mathcal{O}_{n}^{\ast})\subseteq (K^{\ast },\mathcal{O}^{\ast })$
such that $\varphi _{n}(x_{i})=y_{i}$ for $i\leq n.$ To extend $\varphi _{n}$
to $\varphi _{n+1}$ we prove first that the order type of $x_{n+1}$ over $%
H_{n}$ maps by $\varphi _{n}$ to the type of $y_{n+1}$ over $\varphi
_{n}(H_{n}).$ Since $x_{n+1}$ and $y_{n+1}$ are
positive unit in $\mathcal{O}$, $\mathcal{O}^{\ast },$ respectively, it is
enough to consider parameters $h\in (\mathcal{O}_{H_{n}}^{\times })^{>0}$,
which we have seen have the form $h=b(1+\epsilon
)\prod\limits_{i=m}^{n}x_{i}^{q_{i}}$ with $b\in \mathbf{(}\mathcal{O}%
_{H}^{\times }\mathcal{)}^{>0}$, $m\leq n$, and $q_{m}\not=0.$ Then

$h>x_{n+1}\Leftrightarrow \log h>\log x_{n+1}\Leftrightarrow \log b+\log(1+\epsilon )+\sum\limits_{i=m}^{n}q_{i}\log x_{i}>%
\log x_{n+1}$

$\Leftrightarrow \log b+\log(1+\epsilon
)+\sum\limits_{i=m+1}^{n+1}q_{i}\log x_{i}>-q_{m}\log x_{m}$ (with 
$q_{n+1}=-1)$

$\Leftrightarrow q\log b+q\log(1+\epsilon
)+\sum\limits_{i=m+1}^{n+1}qq_{i}\log x_{i}>\log x_{m}$ (with $%
q=-q_{m}^{-1})$

$\Leftrightarrow q\log b+q\log(1+\epsilon
)+\sum\limits_{i=m+1}^{n+1}qq_{i}\log x_{i}-a_{m}>\varepsilon x_{m+1}$
\ $(\ast )$

$\Leftrightarrow q\log b-a_{m}>\varepsilon x_{m+1}$ \ ($\ast \ast $)
(justified later)

$\Leftrightarrow q\log b>\log x_{m}$

$\Leftrightarrow b^{q}>x_{m}$

$\Leftrightarrow \varphi (b)^{q}>\varphi _{n}(x_{m})=y_{m}$ (since $m\leq n)$

$\Leftrightarrow \varphi _{n}(h)>y_{n+1}.$

\noindent
The last equivalence holds because $\varphi _{n}(h)=\varphi
_{n}(b)(1+\varphi _{n}(\epsilon ))\prod\limits_{i=m}^{n}y_{i}^{q_{i}}$, and
mimicking the above arguments we have $\varphi
_{n}(h)>y_{n+1}\Leftrightarrow \varphi _{n}(b)^{q}>y_{m}$. To justify the equivalence of $(\ast )\Leftrightarrow (\ast \ast )$, calling 
$\alpha =q\log b-a_{m},$ $\beta =q\log(1+\epsilon
)+\sum\limits_{i=m+1}^{n+1}qq_{i}\log x_{i},$ and noticing that $%
w(\beta )\geq $ $w(\log x_{m+1})>w(x_{m+1}),$ we have

\medskip

$\alpha +\beta >\varepsilon x_{m+1}$ $(\ast )$

$\Rightarrow w(\alpha +\beta )\leq w(x_{m+1})<w(\beta )$

$\Rightarrow w(\alpha +\beta )=w(\alpha )\leq w(x_{m+1})$

$\Rightarrow w(\alpha )<w(x_{m+1}),$ since $w(x_{m+1})\not\in w(H^{\times})$

$\Rightarrow \alpha >x_{m+1}$ $(\ast \ast )$

$\Rightarrow w(\alpha )<w(x_{m+1})$

$\Rightarrow w(\alpha +\beta )=w(\alpha )<w(x_{m+1})$

$\Rightarrow \alpha +\beta >\varepsilon x_{m+1}$ $(\ast )$.

\medskip
\noindent
By o-minimality of $T_{an},$ there is a $L_{an}$-embedding $\varphi
_{n+1}:H_{n}\langle x_{n+1}\rangle \overset{\backsimeq}{\rightarrow} H_{n}^{\ast }\langle
y_{n+1}\rangle \subseteq K^{\ast }$ extending $\varphi _{n}$ such that $%
\varphi _{n+1}(x_{n+1})=y_{n+1}.$ Since $y_{n+1}$ belongs to $\varphi_{n+1}(\O_{n+1})\subseteq \O^* \cap H^*_n\langle y_{n+1}\rangle$, it is a unit in the first ring, and $w(y_{n+1})\not\in w(H^{*\times}_n)$, we obtain by Lemma \ref{Lemma A} that $\varphi _{n+1}(\mathcal{O}_{n+1})=\O^*\cap H^*_n\langle
y_{n+1}\rangle$ and $\varphi _{n+1}$ is a ($T_{an})_{convex}$-embedding.

Define $\varphi ^{\prime }=\cup _{n}\varphi _{n},$ then $\varphi ^{\prime }(%
\log x_{i})=\varepsilon \varphi ^{\prime }(x_{i+1})+\varphi ^{\prime
}(a_{n})=\varepsilon y_{i+1}+\varphi (a_{n})=\log y_{i},$ and by the
formula for $\log a$ above $\varphi ^{\prime }$ preserves log.

\end{proof}

\medskip

This finishes the proof of Claim 2 and thus we obtain the desired result:

\begin{Tma} \label{modelcomplete}
$T_{an}(\mathcal{O}$-$\exp )$ is model complete.
\end{Tma}

\section{Quantifier elimination,  completeness, weak o-minimality}

For a model $(K,\mathcal{O},\exp )$ of $T_{an}(\mathcal{O}$-$\exp )$ we
denote $\log$ the partial inverse of $\exp $ defined for $x\in (%
\mathcal{O}^{\times })^{>0}$ by 
\begin{equation*}
\log x=y\Leftrightarrow \exp y=x,
\end{equation*}%
and $\log x =0$ otherwise. We call $T_{an}(\mathcal{O}$-$\exp ,\log)$ the corresponding expansion and $T_{an^{\ast }}(\mathcal{O}$-$\exp ,\log)$ the further expansion by $^{-1}$ and $\sqrt[n]{}$ for $n>1$.

By definition, $T_{an^{\ast }}(\mathcal{O}$-$\exp ,\log)$ has a universal axiomatization, except for the axiom $%
\exists x\in K\setminus \mathcal{O}$, thus any substructure of a model of this theory with $\mathcal{O}\not=K$ is again a model. Since it is model complete by Theorem \ref{modelcomplete}, to prove that it has quantifier
elimination\ it is enough to check that each substructure of a model of the
theory has a $T_{an^{\ast }}(\mathcal{O}$-$\exp ,\log)$-\emph{closure}
in the following sense:

\begin{Lema}\label{lema18}
\label{closure} Let $(E,\mathcal{O}_{E},\exp _{E},\log_{E})$ be a
substructure of a model $(K,\mathcal{O},\exp ,\log)$ of $T_{an^{\ast
}}(\mathcal{O}$-$\exp ,\log)$. There is a model $(F,\mathcal{O}%
_{F},\exp _{F},\log_{F})$ of $T_{an^{\ast }}(\mathcal{O}$-$\exp,\log)$ extending $(E,\mathcal{O}_{E}\exp _{E},\log_{E})$, such that $(F,\mathcal{O}_{F},\exp _{F},\log_{F})$ can be embedded
over $(E,\mathcal{O}_{E},\exp _{E},\log_{E})$ into every model of $%
T_{an^{\ast }}(\mathcal{O}$-$\exp ,\log)$ extending $(E,\mathcal{O}%
_{E},\exp _{E},\log_{E}).$

\end{Lema}

\begin{proof} If $E\neq \mathcal{O}_{E}$, then in fact $(E,\mathcal{O}%
_{E},\exp _{E},\log_{E})$ is a model of $T_{an^{\ast }}(\mathcal{O}$-$%
\exp ,\log)$ and we finish. Otherwise, pick $x\in K$, $x>E$, and define $F$ as the $L_{an^{\ast }}$-substructure $E\langle x\rangle $
of $K$ generated by $E(x)$ and $\mathcal{O}_{F}$ as the convex
closure of $E$ in $F$. Clearly, $x\not\in \O_F$, and $E<x<|E\setminus E|$, thus $E=\O_F\cap E\langle x\rangle$ by \cite[Lemma 3.7 and Remark 3.8]{Van10}. Hence, $(E,E)\subseteq (F,\O_F)$.

To see that $\mathcal{O}_{F}$ is exp and $\log$-closed in $K$, given $z\in \mathcal{O}%
_{F}$ find $a,b\in E$ such that $a<z<b,$ then $\exp a<\exp z<\exp b$
and since $\exp a,\exp b\in E$ then $\exp z\in \mathcal{O}_{F}.$ If $%
z\in (\mathcal{O}_{F}{}^{\times })^{>0}$ then $a,b$ above may be chosen
positive in $E$ and thus we have $\log a<\log z<\log b$
and thus $\log z\in \mathcal{O}_{F}.$ This shows that, $(F,\mathcal{O}%
_{F},\exp |_{F},\log|_{F})$ is a model of $T_{an}(\mathcal{O}$-$\exp ,%
\log)$, extending $\mathcal{(}E,E,\exp _{E},\log_{E}).\,\ $It
remains to see that it may be embeds in any model $(K^{\ast },\mathcal{O}%
^{\ast },\exp ^{\ast },\log^{\ast })$ extending $(E,E,\exp _{E},\
log_{E})$.

Pick any $y\in K^{\ast}$, $y>\O^*$, then $y>E$ and, clearly, $x$ and $y$ realize the same cut over $E$. Thus, the $%
L_{an^{\ast }}$-substructure $F^{\prime }=E\langle y\rangle $ of $K^{\ast }$
defined as before is $L_{an^{\ast }}$-isomorphic to $E\langle x\rangle $ by
an isomorphism $\phi :F\rightarrow F^{\prime }$ fixing $E$ which must send $%
\mathcal{O}_{F}$ to $\mathcal{O}_{F^{\prime }}$ (being both convex closures
of $E).$ Moreover, $y\not\in \O_{F'}$; therefore, by \cite[Lemma 3.7]{Van10} again, and knowing that $y\not\in \O^*$ then $\O_F=\O^*\cap E\langle y\rangle$ and thus $(F',\O_{F'})\subseteq (K^*,\O^*)$.

To see that $\phi $ preserves $\exp $ or $\log,$ notice that
by Corollary \ref{divisible hull} we have $w(F^{\times })=w(E^{\times })\oplus 
\mathbb{Q}w(x)$. Therefore, for each $a\in \O_{F}^{\times }$, there are $b\in
E$, $\epsilon \in F$ with $w(\epsilon )>0$ and $q\in \mathbb{Q}$ such that $%
a=b(1+\epsilon )x^{q}$. Then since $\log x=\log y=0,$ $\log b\in E$, and $\phi $ preserves $E$ and restricted analytic functions,%
\begin{eqnarray*}
\phi (\log a) &=&\phi (\log b)+\phi (\log(1+\epsilon ))=%
\log^{\ast } b +\log^{\ast }(1+\phi (\epsilon )) \\
&=&\log^{\ast }(\phi (b(1+\epsilon )y^{q}))=\log^{\ast } \phi
(a).
\end{eqnarray*}%
Therefore, 
\begin{equation*}
\phi :(F,\mathcal{O}_{F},\exp |_{F},\log|_{F})\approx (F^{\prime },%
\mathcal{O}_{F^{\prime }},\exp ^{\ast }|_{F^{\prime }},\log^{\ast
}|_{F^{\prime }})\subseteq (K^{\ast },\mathcal{O}^{\ast },\exp ^{\ast },%
\log^{\ast })
\end{equation*}%
over $(E,\mathcal{O}_{E},\exp _{E},\log_{E})$.

\end{proof}

\begin{Tma}
\label{quantif} The theory $T_{an^{\ast }}(\mathcal{O}$-$\exp ,\log)$
has quantifier elimination.
\end{Tma}

Now, by Lemma \ref{lemanuevo9} $ (\mathbb{R}_{an^{\ast }},\mathbb{R},\exp ,\log)$ is a
substructure of any model of $T_{an^{\ast }}(\mathcal{O}$-$\exp ,\log)$; hence, by the proof of Lemma \ref{lema18}, its closure $(\mathbb{R}_{an^{\ast }}\langle x\rangle,\mathcal{O}_{\mathbb{R}},\exp,\log)$, where $\mathcal{O}_{\mathbb{R}}$ is the natural valuation ring of $\mathbb{R}%
_{an^{\ast }}\langle x\rangle $, is a prime model of $T_{an^{\ast }}(%
\mathcal{O}$-$\exp ,\log)$. Since ${}^{-1}$, $\sqrt[n]{}$, and $\log$
are definable, its reduct to the language $L_{an,\mathcal{O},\exp }$ is
a prime model of $T_{an}(\mathcal{O}$-$\exp )$; hence,

\begin{Tma}
\label{complete}$T_{an}(\mathcal{O}$-$\exp )$ is complete, with prime model $(\mathbb{R}_{an}\langle x\rangle,\O_{\mathbb{R}},\exp)$.
\end{Tma}

$T_{an}(\mathcal{O}$-$\exp )$ can not be o-minimal since the valuation ring
can not have frontier points, but we know $(T_{an})_{convex}$ is weakly
o-minimal. We may extend this to $T_{an}(\mathcal{O}$-$\exp ).$ Notice first
the following immediate but somehow unexpected consequence of model
completeness and the existence of a prime model.

\begin{Lema}\label{exttotal}
Any model of $T_{an}(\mathcal{O}$-$\exp )$ may be embedded in a model where
the partial exponential of the valuation ring may be extended to a total
exponential of the field.
\end{Lema}

\begin{proof} Let $(K,\mathcal{O},\exp )$ be a model of $T_{an}(\mathcal{O%
}$-$\exp )$ and let $(K^*,\exp^*)$ be any $|K|^{+}$-saturated model
of $T_{an}(\exp )$. Consider the expansion $%
(K^*,\exp^*,\mathcal{V})$ where $\mathcal{V}$ is the natural
valuation ring of $K^*$, then $\exp^*(\mathcal{V})=(\mathcal{V}%
^{\times })^{>0}$ and thus  $(K^{\ast},\mathcal{V},\exp ^{\ast}\upharpoonright \mathcal{V})$ is a $|K|^{+}$-saturated model of $T_{an}(\mathcal{O}$-$\exp ).$ As the
prime model of $T_{an}(\mathcal{O}$-$\exp )$ embeds in both models, there must
exist, by the model completeness criterium used above, an embedding $%
\varphi :$ $(K,\mathcal{O},\exp )\rightarrow (K^{\ast},\mathcal{V},\exp ^{\ast}\upharpoonright \mathcal{V}),$
which proves the claim.
\end{proof}

\medskip

Combining the previous observation with the result of Baizhanov 
\cite{Baijanov}, \cite[Theorem 63]{Baijanov2} that a o-minimal structure expanded by a convex predicate is weakly o-minimal, we obtain

\begin{Tma}
\label{weakly} The theory $T_{an}(\mathcal{O}$-$\exp )$ is weakly o-minimal.
\end{Tma}

\begin{proof} Let $(K,\mathcal{O},\exp )$ be a model of $T_{an}(\mathcal{O%
}$-$\exp ),$ the reduct $(K^{\ast},\exp ^{\ast})$ of the
extension obtained in Lemma \ref{exttotal} is a model of $T_{an}(\exp )$ and thus it is
o-minimal; therefore, $(K^{\ast},\exp ^{\ast},\mathcal{V})$ is weakly o-minimal by Baizhanov result, and a fortiori
the same is true of $(K^{\ast},\mathcal{V},\exp ^{\ast}\upharpoonright \mathcal{V}).$ But this property is
inherited by elementary substructures; therefore, $(K,\mathcal{O},\exp )$ is
weakly o-minimal.

\end{proof}


Finally, notice that if we add a new constant symbol $c$ to the language, the theory 
\begin{equation*}
T_{an^{\ast }}(\mathcal{O}\text{-}\exp ,\log,c)=T_{an^{\ast }}(%
\mathcal{O}\text{-}\exp ,\log)\cup \{c>0\wedge c\notin \mathcal{O}\}
\end{equation*}%
has a universal axiomatization and is still model complete. Therefore it admits definable Skolem functions by a trivial application of a result of P. Scowcroft (\cite[Theorem 1]{Scow}). By
definability of $\ ^{-1},$ $\sqrt[n]{},$ and $\log$ in $T_{an}(\mathcal{O}$%
-$\exp ,c)$ the same is true of the later theory.

\section{Final observation}

Theorem \ref{complete} implies, after Example 3, that $T_{an}(\O$-$\exp)=Th(\mathbb{R}((t^{\Gamma}))_{an},\O_w,\exp)$ for any Hahn field with divisible group $\Gamma$ and its natural valuation ring $\O_w$. Consider the theory of fields of the form $(\mathbb{R}((t^{\Gamma})),\O_w,\exp)$, where $\Gamma$ is not necessarily divisible, as a valued field with partial exponentiation. It is reasonable to expect an Ax-Kochen-Ershov type result for this kind of structures.

\medskip

\noindent \textbf{Acknowledgements.} This paper covers some results obtained in the PhD-thesis of the first author. The authors thanks Lou van den Dries for his suggestion to study this theory and his helpful remarks, and the referee for his criticism and comments to improve the paper. 

\medskip
\section*{Declarations}
\noindent \textbf{Funding.} No funding was received to assist with the preparation of this manuscript.







\begin{thebibliography}{99}
\bibitem{van0} Aschenbrenner, M., van den Dries, L., van der Hoeven, J.: Towards a model theory for transseries. Notre Dame J. Form. Log. 54, no. 3-4, 279-310 (2013).
\bibitem{Van6} Aschenbrenner, M., van den Dries, L., van der Hoeven, J.: 
Asymptotic Differential Algebra and Model Theory of Transseries.
Ann. of Math. Stud. 195, Princeton University Press (2017)
\bibitem{Baijanov} Baizhanov, B. S.: Expansion of an o-minimal model by unary convex predicates. Researches in theory of algebraic systems (Nurmagambetov, T., editor, in Russian), Karaganda State University, 3-23 (1995).

\bibitem{Baijanov2} Baizhanov, B. S.: Expansion of a model of a weakly o-minimal theory by a family of unary predicates. J. Symbolic Logic 66(3), 1382-1414 (2001)


\bibitem{Chang} Chang, C., Keisler, H.: Model theory, (3rd
ed.). Studies in Logic and the Foundation of Mathematics. North-Holland,
Amsterdam, 73 (1990)

\bibitem{Van9} Denef J., van den Dries, L.: $p$-adic and real
subanalytic sets. Ann. of Math. 128, 79-138 (1988)

\bibitem{Dickmann} Dickmann, M.: Elimination of quantifiers for
ordered valuation rings. J. Symbolic Logic 52, 116-128 (1987)

\bibitem{End} Endler, O.: Valuation theory. Springer-Verlag, New
York (1972)

\bibitem{Eng} Engler, A., Prestel A.: Valued Fields. Springer Monographs in Mathematics, (2205)

\bibitem{Gabrielov} Gabrielov, A.: Projections of semi-analytic sets.
Functional Anal. Appl. 2, 282-291 (1968)

\bibitem{Haskell} Haskell, D.: Model theory of analytic functions: Some
historical comments. Bulletin of Symbolic Logic 18(3),368-381 (2012)

\bibitem{Hod} Hodges, W.: Model theory. Encyclopedia of Mathematics
and its Applications. Cambridge University Press 42, (1993)

\bibitem{Kuh4} Kuhlmann, S., On the structure of nonarchimedean exponential fields I. Arch. Math. Logic 34, 145-182 (1995)

\bibitem{Kuhlibro} Kuhlmann, S.: Ordered exponential fields. The Fields Institute Monograph Series, Vol. 12, AMS Publications, 164 pages (2000).

\bibitem{Kuhshe} Kuhlmann, F., Kuhlmann, S., Shelah, S.: Exponentiation in power series fields. Proc. Amer. Math. Soc. 125, 3177-3183 (1997)

\bibitem{Kuh5} Kuhlmann, F., Kuhlmann, S.: The exponential rank of
nonarchimedean exponential fields. Delzell and Madden (eds): Real algebraic
geometry and ordered structures. Contemp. Math. 253, 181-201 (2000)



\bibitem{P-S} Pillay, A., Steinhorn, C.: Definable sets in ordered
structures. Bull. of the Amer. Math. Soc. 11(1), 159-162 (1984)

\bibitem{sacks} Sacks, G.: Saturated model theory. W. A. Benjamin,
Inc., Reading, Mass. (1972)

\bibitem{Scow} Scowcroft, P.: A note on Definable Skolem Functions.
J. Symbolic Logic 53(3), 905-911 (1988)




\bibitem{Van8} van den Dries, L.: A generalization of the
Tarski-Seidenberg theorem, and some nondefinability results. Bull. AMS. 15, 189-193 (1986)


\bibitem{Van1} van den Dries, L., Macintyre, A., Marker, D.: The
Elementary Theory of Restricted Analytic Fields with Exponentiation. Annals
of Mathematics 140(1), 183-205 (1994)

\bibitem{Van10} van den Dries, L., Lewenberg, A.: T-Convexity and Tame Extensions. The Journal of Symbolic Logic 60(1), 74-102 (1995)


\bibitem{Van2} van den Dries, L., Macintyre, A., Marker, D.: Logarithmic-Exponential Power Series. J. London Math. Soc. 56(2), 417-434 (1997)


\bibitem{main} van den Dries, L.: Lectures on the model theory of
valued fields. Lecture Notes in Mathematics 2111, 55-157 (2014)

\bibitem{Wilk1} Wilkie, A.: Model completeness results for
expansions of the ordered field of real numbers by restricted Pfaffian
functions and the exponential function. J. Amer. Math. Soc. 9(4), 1051-1094 (1996)
\end{thebibliography}

\end{document}